\documentclass[12pt]{amsart}
\usepackage{amsthm,amsmath,amssymb,tikz-cd,enumerate}
\usepackage{graphicx}
\pdfoutput=1
{\end{list}}
{
   \newtheorem{theorem}{Theorem}[section]
   
   \newtheorem{lemma}[theorem]{Lemma}

   \newtheorem{coro}[theorem]{Corollary}

}
{\theoremstyle{definition}
   \newtheorem{rema}[theorem]{Remark}

}

\newcommand{\RR}{{\mathbb{R}}}
\newcommand{\FF}{{\mathbb{F}}}

\newcommand{\QQ}{{\mathbb{Q}}}

\newcommand{\ZZ}{{\mathbb{Z}}}

\newcommand{\one}{{\mathbf{1}}}

\newcommand{\cA}{{\mathcal A}}

\newcommand{\Gr}{\operatorname{Gr}}

\newcommand{\Link}{\operatorname{Link}}

\newcommand{\Spec}{\operatorname{Spec}}

\newcommand{\Star}{\operatorname{Star}}

\newcommand{\Tor}{\operatorname{Tor}}

\newcommand{\isom}{\simeq}

\newcommand{\SL}{{\operatorname{SL}}}

\newcommand{\ev}{{\operatorname{ev}}}

\newcommand{\Po}{Poincar\'e }

\newcommand{\tensor}{\otimes}

\renewcommand{\a}[1]{a_{#1}}
\newcommand{\ap}[1]{a'_{#1}}

\begin{document}
\title{On the anisotropy theorem of Papadakis and Petrotou}

\author{Kalle Karu}

\author{Elizabeth Xiao}
\thanks{This work was partially supported by an NSERC Discovery grant.}
\address{Mathematics Department\\ University of British Columbia \\
  1984 Mathematics Road\\
Vancouver, B.C. Canada V6T 1Z2}
\email{karu@math.ubc.ca, elizabeth@math.ubc.ca}
% \date{Dec 1, 2001}


\begin{abstract}
We study the anisotropy theorem for Stanley-Reisner rings of simplicial homology spheres in characteristic $2$ by Papadakis and Petrotou. This theorem implies the Hard Lefschetz theorem as well as McMullen's $g$-conjecture for such spheres. Our first result is an explicit description of the quadratic form. We use this description to prove a conjecture stated by Papadakis and Petrotou. All anisotropy theorems for homology spheres and pseudo-manifolds  in characteristic $2$ follow from this conjecture. Using a specialization argument, we prove anisotropy for certain  homology spheres over the field $\QQ$. These results provide another self-contained proof of the $g$-conjecture for homology spheres in characteristic $2$.
\end{abstract}

\maketitle

\section{Introduction}

McMullen's g-conjecture \cite{McMullen-g}  characterizes all possible face numbers of simplicial polytopes $\Delta$. The sufficiency part of the conjecture was proved by Billera and Lee \cite{BilleraLee}. Stanley \cite{StanleyHL} proved the necessity by applying the Hard Lefschetz theorem to the cohomology ring $H(\Delta)$. The Hard Lefschetz theorem is traditionally proved together with the Hodge-Riemann bilinear relations, which state that a quadratic form is positive definite on the primitive cohomology. Since the ground field is assumed to be $\RR$, this is equivalent to the quadratic form on the primitive cohomology being anisotropic with the positive sign.

When trying to generalize the g-conjecture from simplicial polytopes to simplicial homology spheres, one is faced with the fact that there is no convexity and hence no positivity for the Hodge-Riemann relations. Proving Hard Lefschetz without Hodge-Riemann  relations is very hard (e.g. see \cite{Adiprasito}). However, in order to deduce Hard Lefschetz from Hodge-Riemann relations, one does not need positivity. Indeed, anisotropy of the quadratic form is sufficient. Papadakis and Petrotou \cite{PapadakisPetrotou} prove a very strong version of anisotropy of the quadratic form on not just the primitive cohomology but the whole middle degree cohomology. This theorem is the motivation for all results in the current article.

The theorem of Papadakis and Petrotou applies to simplicial homology spheres $\Delta$ over a field $k$ of characteristic $2$. The Stanley-Reisner ring $\cA(\Delta)$ and the cohomology ring $H(\Delta) = \cA(\Delta)/(\theta_1,\ldots,\theta_n)$ are defined over a larger field $K=k(\a{i,j})$ of rational functions in the variables $\a{i,j}$. The variables $\a{i,j}$ here are the coefficients of the linear parameters $\theta_1,\ldots, \theta_n$ in the definition of $H(\Delta)$. 

\begin{theorem}[Papadakis, Petrotou] \label{thm-PP}
Let $\Delta$ be a simplicial homology sphere of dimension $n-1=2m-1$ over a field $k$ of characteristic $2$. Let $H(\Delta)$ be defined over the field of rational functions $K=k(\a{i,j})$. Then the quadratic form defined on the middle degree cohomology $H^m(\Delta)$ by multiplication 
\[ Q(g) = g^2 \in H^{n}(\Delta) \isom K\]
 is anisotropic.
\end{theorem}

Papadakis and Petrotou used Theorem~\ref{thm-PP}  to prove the Hard Lefschetz theorem for all simplicial homology spheres in characteristic $2$,  in both even and odd dimensions. The Hard Lefschetz theorem then implies the g-conjecture for such spheres. 

Our first result in this article is an explicit description of the quadratic form $Q$ that holds in any characteristic (Theorem~\ref{thm-Q} below). We use this description to prove in Theorem~\ref{thm-PPconj} a conjecture stated in \cite{PapadakisPetrotou} that generalizes the main ingredient in the proof of Theorem~\ref{thm-PP}. As an application of the conjecture, we prove anisotropy in all degrees $m\leq n/2$.

One can define the Hodge-Riemann type quadratic form in any degree. Let
\[ l = x_1 + x_2 + \ldots + x_N \in H^1(\Delta),\]
where $\Delta$ has $N$ vertices and $x_1,\ldots, x_N$ are the corresponding variables in the Stanley-Reisner ring. 
The quadratic form $Q_l$ on $H^m(\Delta)$ for $m\leq n/2$ is defined by
\[ Q_l(g) = l^{n-2m} g^2 \in H^n(\Delta) \isom K.\]
%This form is defined for both even and odd $n$. 

\begin{theorem} \label{thm-lef}
Let $\Delta$ be a simplicial homology sphere of dimension $n-1$ over a field $k$ of characteristic $2$, and let $H(\Delta)$ be defined over the field of rational functions $K=k(\a{i,j})$. Then the quadratic form $Q_l$ is anisotropic on $H^m(\Delta)$ for any $m\leq n/2$. 
\end{theorem}

Theorem~\ref{thm-lef} can be deduced from Theorem~\ref{thm-PP} using induction on the dimension of the sphere and Hard Lefschetz theorem \cite{PapadakisPetrotou}. However, Theorem~\ref{thm-lef} is also a simple application of the conjecture in \cite{PapadakisPetrotou}. 
%The conjecture is indeed well suited for proving all types of anisotropy results in characteristic $2$. 
Note also that Theorem~\ref{thm-lef} directly implies the Hard Lefschetz theorem, which is equivalent to the form $Q_l$ being nondegenerate on $H^m(\Delta)$ for $m \leq n/2$.

The explicit description of the quadratic form allows us to use a specialization argument to show that anisotropy in characteristic $2$ implies the same in characteristic $0$ over the field $k=\QQ$. 

\begin{theorem} \label{thm-2to0}
Let $\Delta$ be a simplicial homology sphere over the field $\FF_2$. Then Theorem~\ref{thm-lef} holds when $H(\Delta)$ is defined over the field $K_0 = \QQ(\a{i, j})$. 
\end{theorem}

Let us denote by $HS(R)$ the set of simplicial homology spheres over a coefficient ring $R$ (see Section~\ref{sec-SR} for definition). Then we have a sequence of  inclusions
\[ \{\text{Topological spheres}\} \subset HS(\ZZ) \subset HS(\FF_2) \subset HS(\QQ).\]
All theorems stated above apply to homology spheres over $\FF_2$, hence they also apply to topological spheres and integral homology spheres. The cohomology ring $H(\Delta)$ defined over a field $K_0$ of characteristic $0$ is well-behaved when $\Delta$ is a homology sphere over $\QQ$, but the anisotropy problem in this case remains open.
 
 The conjecture in \cite{PapadakisPetrotou} and the anisotropy theorems are more naturally stated for pseudo-manifolds $\Delta$. Theorem~\ref{thm-PP} for pseudo-manifolds in characteristic $2$ was proved by  Adiprasito, Papadakis and Petrotou \cite{APP}. We will work everywhere below in the generality of pseudo-manifolds.

\subsection{Outline of the article} Our main tool in the proofs of anisotropy is the mixed volume $W_\Delta$. This is the linear function on the space of degree $n$ homogeneous polynomials:
\[ W_\Delta: K[x_1,x_2,\ldots,x_N]_n \to H^n(\Delta) \stackrel{\isom}{\longrightarrow} K.\]
The mixed volume determines the ring $H(\Delta)$ if $\Delta$ is a simplicial homology sphere, and in particular it determines the quadratic form $Q_l$ on  $K[x_1,x_2,\ldots,x_N]_m$:
\[ Q_l(g) = W_\Delta(l^{n-2m} g^2).\]
 In Section~\ref{sec-mixed} we prove a decomposition theorem for mixed volumes. If $\Delta$ decomposes as a connected sum, $\Delta = \Delta_1 \# \Delta_2$, then the mixed volume also decomposes, 
\[ W_\Delta = W_{\Delta_1}+ W_{\Delta_2}.\]
 We decompose $\Delta = \#_{i=1}^M \Pi_i$ into a connected sum where each $\Pi_i$ is the boundary sphere of an $n$-simplex. This decomposition provides an explicit diagonal formula for the quadratic form $Q_l$ that is valid in any characteristic. In Section~\ref{sec-consec} we use the formula to specialize the quadratic form from characteristic $0$ to characteristic $2$.  
 
The decomposition of the mixed volume is compatible with the conjecture of Papadakis and Petrotou, reducing the conjecture to the case of $\Pi_i$. We prove the conjecture and the anisotropy theorems in Section~\ref{sec-conjecture}.

We start the next section by recalling the definitions of simplicial homology spheres, Stanley-Reisner rings, and Brion's construction of the isomorphism $H^n(\Delta) \isom K$.

\section{Stanley-Reisner rings}
\label{sec-SR}

We work over a field $k$ of any characteristic in this section.
Let $\Delta$ be a (finite, abstract) simplicial complex of dimension $n-1$. We write $\Delta_d$ for the set of $d$-dimensional simplices of $\Delta$. The complex $\Delta$ is called pure if all its maximal simplices have the same dimension $n-1$, which we call the dimension of $\Delta$.

\subsection{Homology spheres} A pure simplicial complex $\Delta$ of dimension $n-1$ is a homology sphere over a coefficient ring $R$ if for every simplex $\tau \in \Delta$ the link of $\tau$ has the same reduced homology as a sphere of dimension $n-2-\dim \tau$:
\[ \tilde{H}_i (\Link\tau; R) = \begin{cases} R & \text{if $i=n-2-\dim\tau$},\\
0 & \text{otherwise.} \end{cases}
\]
 The homology here is the simplicial homology with coefficients in the ring $R$. The condition also needs to hold for the empty simplex that has dimension $-1$. 
 
Stanley-Reisner rings are defined over a field, and the theory works best for homology spheres $\Delta$ over the same field (in this case the algebra $H(\Delta)$ is Gorenstein by Reisner's theorem). The condition for a simplicial complex $\Delta$ to be a homology sphere over a field $k$ only depends on the characteristic of the field and not on the field itself. We clarify here the relationship between homology spheres over different coefficient rings. The following result is elementary.

\begin{lemma}
Let $\Delta$ be a pure simplicial complex of dimension $n-1$. 
\begin{enumerate}
\item If $\Delta$ is a homology sphere over $\ZZ$, then it is a homology sphere over any ring $R$.
\item If $\Delta$ is a homology sphere over $\FF_p$ for some prime $p$, then it is a homology sphere over $\QQ$.
\end{enumerate}
\end{lemma}

\begin{proof}
We will consider the homology $H_i(\Link\tau; R)$ when $\tau=\emptyset$, $\Link\tau=\Delta$. The case of general $\tau$ is similar.

The first statement follows from the universal coefficient theorem which gives an exact sequence
\[ 0\to \tilde{H}_i (\Delta; \ZZ) \tensor R \stackrel{\nu}{\to} \tilde{H}_i (\Delta; R) \to \Tor_1(\tilde{H}_{i-1} (\Delta; \ZZ), R) \to 0.\]
If $\Delta$ is a homology sphere over $\ZZ$, then the $\Tor_1$ term vanishes for all $i$. Hence $\nu$ is an isomorphism.

For the second statement, consider the exact sequence
\[  0\to \tilde{C}_\cdot(\Delta;\ZZ) \stackrel{\mu}{\to} \tilde{C}_\cdot(\Delta;\ZZ) \to \tilde{C}_\cdot(\Delta;\FF_p) \to 0,\]
where $\tilde{C}_\cdot(\Delta;R)$ is the augmented simplicial chain complex with coefficients in $R$, and the map $\mu$ is multiplication by $p$. The short exact sequence of complexes gives a long exact sequence of homology groups. For $i< n-2$ we get an isomorphism
\[ \tilde{H}_i(\Delta;\ZZ) \stackrel{\mu}{\to} \tilde{H}_i(\Delta;\ZZ).\]
Since the homology groups are finitely generated abelian groups, it follows that $\tilde{H}_i(\Delta;\ZZ)$ is a finite group with no $p$-torsion. 

For $i=n-1$ we get an exact sequence 
\[  0\to \tilde{H}_i(\Delta;\ZZ) \stackrel{\mu}{\to} \tilde{H}_i(\Delta;\ZZ) \to \tilde{H}_i(\Delta;\FF_p) \to \tilde{H}_{i-1}(\Delta;\ZZ) \stackrel{\mu}{\to} \tilde{H}_{i-1}(\Delta;\ZZ) \to 0.\]
The right map $\mu$ being surjective implies that $\tilde{H}_{i-1}(\Delta;\ZZ)$ is a finite abelian group with no $p$-torsion. In particular, the right map $\mu$ is an isomorphism. The group $\tilde{H}_i(\Delta;\ZZ)$ is a subgroup of the free abelian group $\tilde{C}_i(\Delta,\ZZ)$, and hence is itself a free abelian group. Since $\tilde{H}_i(\Delta;\FF_p) = \FF_p$, we get $\tilde{H}_i(\Delta;\ZZ) = \ZZ$.  In summary, the integral reduced homology of $\Delta$ is $\ZZ$ in top degree and a finite abelian group in lower degrees. Now the universal coefficient theorem shows that the  homology groups with $\QQ$ coefficients are as required.
\end{proof}

\subsection{Pseudo-manifolds}
Homology spheres are a special case of pseudo-manifolds. A pseudo-manifold is a pure simplicial complex of dimension $n-1$ such that
\begin{enumerate}[(a) ]
\item Every  $(n-2)$-simplex  lies in exactly two $(n-1)$-simplices.
\item $\Delta$ is strongly connected: the geometric realization of $\Delta$ remains connected after we remove its $(n-3)$-skeleton. 
\end{enumerate}
If we allow every $(n-2)$-simplex to lie in either one or two $(n-1)$-simplices, then we obtain a pseudo-manifold with boundary. The $(n-2)$-simplices that lie in only one $(n-1)$-simplex generate a subcomplex $\partial \Delta$ called the boundary. In the following, by a pseudo-manifold we always mean a pseudo-manifold with empty boundary.

For a pseudo-manifold (with or without boundary) it makes sense to talk about orientability. An orientation on a simplex is an ordering of its vertices, up to changing the ordering by an even permutation. An orientation $v_{j_1}, v_{j_2}, \ldots, v_{j_n}$ on a simplex induces the orientation $v_{j_2}, \ldots, v_{j_n}$ on its facet. An orientation on a pseudo-manifold is an orientation on all its maximal simplices of dimension $n-1$ such that for every $(n-2)$-simplex that lies in two $(n-1)$-simplices, the orientations induced from the two $(n-1)$-simplices are opposite.

A pseudo-manifold (with or without boundary) of dimension $n-1$ is orientable if and only if the relative homology group $\tilde{H}_{n-1}(\Delta, \partial\Delta; R) = R$ for some ring $R$ in which $-1\neq 1$   (equivalently, for all such rings $R$). We will use this homological condition to define when $\Delta$ is orientable over the field $k$. Then over a field of characteristic $2$ every pseudo-manifold is orientable, with an orientation consisting of an arbitrary ordering of vertices of each maximal simplex.

Every homology sphere over $k$ is orientable over $k$, because the condition $\tilde{H}_{n-1}(\Delta; k) = \tilde{H}_{n-1}(\Link \emptyset; k) = k$  is part of the definition of homology sphere. 

Oriented pseudo-manifolds over a field $k$ are the most general simplicial complexes that we will consider below. We will state all results for such complexes and sometimes mention the special case of homology spheres.

\subsection{Stanley-Reisner rings}

Let $\Delta$ be a simplicial complex of dimension $n-1$  with vertices $v_1, \ldots, v_N$. The Stanley-Reisner ring of  $\Delta$ over a field $K$ is
\[ \cA(\Delta) = K[x_1,\ldots, x_N]/I_\Delta,\]
where $I_\Delta$ is the ideal generated by all square-free monomials $\prod_{i\in S} x_i$ such that the set $\{ v_i \}_{i\in S}$ is not a simplex in $\Delta$. The ring $\cA(\Delta)$ is a graded $K$-algebra. Given homogeneous degree $1$ elements $\theta_1,\ldots, \theta_n \in \cA^1(\Delta)$, we define the cohomology ring
\[ H(\Delta) = \cA(\Delta)/(\theta_1,\ldots,\theta_n).\]
To remove dependence on the choice of $\theta_i$, we work with generic parameters
\[ \theta_i = \a{i, 1} x_1 + \a{i, 2} x_2 + \cdots + \a{i, N} x_N, \quad i=1,\ldots,n,\]
where $\a{i, j}$ are indeterminates and the field $K$ is the field of rational functions $K=k(\a{i, j})$. We will only consider this generic case. 

If $\Delta$ is a homology sphere over $k$ and the parameters are generic as above, then the ring $H(\Delta)$ is a standard graded, Artinian, Gorenstein $K$-algebra of socle degree $n$. The \Po pairing defined by multiplication 
\[ H^m(\Delta) \times H^{n-m}(\Delta) \longrightarrow H^n(\Delta) \isom K\]
is a nondegenerate  bilinear pairing. If $\Delta$ is only an oriented pseudo-manifold over $k$, then we still have $H^n(\Delta) \isom K$, but the pairing may be degenerate.

\subsection{Piecewise polynomial functions}

It follows from a result of Billera \cite[Theorem 3.6]{Billera} that the Stanley-Reisner ring $\cA(\Delta)$ defined over the field $\RR$ is isomorphic to the ring of piecewise polynomial functions on a fan.  The fan here is the simplicial fan with each simplex in $\Delta$ replaced by a convex cone generated by the simplex. A piecewise polynomial function on the fan  $\Delta$ is a collection of polynomial functions $f_\sigma$ on maximal cones $\sigma$ that agree on the intersections of cones. 

The isomorphism between the Stanley-Reisner ring and the ring of piecewise polynomial functions is given as follows.  The rays ($1$-dimensional cones) of the fan are generated by the vertices $v_j$ of $\Delta$. The vertices $v_j$ define marked points on the rays they generate. Each variable $x_i$ defines a piecewise linear function on the fan that is uniquely determined by its values $x_i(v_j) = \delta_{i,j}$. This defines a morphism from the algebra $\RR[x_1,\ldots, x_N]$ to the $\RR$-algebra of piecewise polynomial functions. The kernel of this morphism is the Stanley-Reisner ideal $I_\Delta$. 

The parameters $\theta_1, \ldots, \theta_n$ (with coefficients $\a{i,j} \in\RR$) are piecewise linear functions on the fan. They define  a piecewise linear map $\Delta \to V =\RR^n$. We assume that this map is injective on every cone. If $\sigma$ is an $n$-dimensional cone, then a polynomial function on $\sigma$ is the same as a polynomial function on $V$. Hence for a pure $n$-dimensional fan and fixed parameters $\theta_1,\ldots,\theta_n$, an element $f\in \cA(\Delta)$ is a collection $\{ f_\sigma \}$ of polynomials on $V$,
\[ f_\sigma \in \RR[t_1,\ldots,t_n]\]
such that $f_{\sigma_1}$ and $f_{\sigma_2}$ agree on the image of $\sigma_1\cap \sigma_2$.

The above isomorphism between the Stanley-Reisner ring and the ring of piecewise polynomial functions carries over to the case where the rings are defined over an arbitrary field $K$. The fan is replaces by the affine scheme $\Spec \cA(\Delta)$. This scheme consists of linear $n$-dimensional spaces, one for each maximal simplex $\sigma$, glued along subspaces. The linear parameters $\theta_i$ define a finite morphism from this scheme to the $n$-space $V = \Spec K[t_1,\ldots, t_n]$, and for a pure complex $\Delta$ we may again view an element $f\in \cA(\Delta)$ as a collection of polynomials, one for each maximal simplex $\sigma$, 
\[ f=  \{ f_\sigma \}, \quad f_\sigma \in K[t_1,\ldots, t_n]. \]
The pullback of $t_i$ is $\theta_i$. This turns $\cA(\Delta)$ into a graded $K[t_1,\ldots, t_n]$-module, where $t_i$ acts by multiplication with $\theta_i$.

Let us  find the piecewise polynomial function $\{ f_\sigma \}$ defined by a polynomial $f\in  K[x_1,\ldots,x_N]$. 
Let $\sigma = \{v_{j_1},\ldots,v_{j_n}\}$ be a maximal simplex in $\Delta$. The piecewise linear map $\theta$ gives an isomorphism 
\begin{align} \label{eq-isom}
 K[t_1,\ldots,t_n] &\isom K[x_{j_1},\ldots, x_{j_n}] \\ \nonumber
 t_i &\mapsto  \a{i,  j_1} x_{j_1} + \cdots + \a{i,  j_n} x_{j_n}.
 \end{align} 
(This is the isomorphism between polynomial functions on $V$ and polynomial functions on the $n$-plane corresponding to $\sigma$.) Now given a polynomial $f(x_1,\ldots, x_N)$, we first map it to $K[x_{j_1},\ldots, x_{j_n}]$ by setting all other variables $x_j$ equal to zero. Then we apply the inverse of the isomorphism to get a polynomial $f_\sigma(t_1,\ldots, t_n)$. This construction defines an isomorphism from the Stanley-Reisner ring of $\Delta$ to the ring of piecewise polynomial functions.

\subsection{Brion's integration map}

Brion in \cite{Brion} defined the isomorphism 
\[  H^n(\Delta) \to K \]
in terms of piecewise polynomial functions on the fan $\Delta$. We describe this map in the more general case where the field $K$ is not necessarily  $\RR$, and $\Delta$ is a pseudo-manifold.

The isomorphism depends on a fixed volume form 
\[ t_1\wedge t_2 \wedge \cdots \wedge t_n \in \Lambda^n V^*,\]
and an orientation on $\Delta$ over the field $K$.  

Let $\sigma$ be a maximal simplex in $\Delta$, and let $v_{j_1},\ldots,v_{j_n}$ be an ordering of its vertices given by the orientation. (If the characteristic of $K$ is $2$, then any ordering is allowed.)
Using the isomorphism (\ref{eq-isom}), define the polynomial $\chi_\sigma \in K[t_1,\ldots, t_n]$ as
\[ \chi_\sigma = c_\sigma x_{j_1} x_{j_2} \cdots x_{j_n},\]
where the constant $c_\sigma \in K$ is such that 
\[ c_\sigma x_{j_1} \wedge x_{j_2} \wedge \cdots \wedge x_{j_n} = t_1\wedge t_2 \wedge \cdots \wedge t_n.\]
One can compute that 
 \begin{equation} \label{eq-dets}
  c_\sigma= \det\sigma = \det \begin{bmatrix} 
 \a{1, j_1} & \a{1, j_2} & \ldots & \a{1, j_n} \\
 \vdots & \vdots & \ddots & \vdots \\
  \a{n, j_1} & \a{n, j_2} & \ldots & \a{n, j_n}
  \end{bmatrix}.
  \end{equation}

The following lemma was proved in \cite[Theorem~2.2]{Brion} in the case where the field is $\RR$ and $\Delta$ is a complete fan. The proof for an arbitrary field $K$ and an oriented pseudo-manifold $\Delta$ is the same, so we recall it.  

\begin{lemma} 
Let $\Delta$ be an oriented pseudo-manifold of dimension $n-1$ over $k$. Consider
\begin{align} \label{eq-int}
 \pi_{\Delta}: \cA(\Delta) & \to K(t_1, \ldots, t_n) \\ \nonumber
 f & \mapsto  \sum_{\sigma \in \Delta_{n-1}} \frac{f_\sigma}{\chi_\sigma}.
 \end{align} 
Then the image of $\pi_{\Delta}$ lies in $K[t_1, \ldots, t_n]$ and the induced map
\[ \pi_{\Delta}: \cA(\Delta)  \to K[t_1, \ldots, t_n]\]
is a degree $-n$ homomorphism of graded $K[t_1,\ldots, t_n]$ modules. The map $\pi_{\Delta}$ in degree $n$ defines an isomorphism
\begin{equation*} 
 \pi_{\Delta}: H^n(\Delta) \to K.
 \end{equation*}
\end{lemma}

\begin{proof}
We start by proving the last statement of the lemma. Recall that $t_i$ act on $\cA(\Delta)$ by multiplication with $\theta_i$. It follows that $\pi_{\Delta}$ maps $(\theta_1,\ldots,\theta_n) \cA^{n-1}(\Delta)$ to zero because it maps $\cA^{n-1}(\Delta)$ to elements of degree $-1$ in $K[t_1,\ldots, t_n]$. Hence $\pi_{\Delta}$ in degree $n$ factors through $H^n(\Delta)$. This map is nonzero because the piecewise polynomial function $\{ f_\sigma\}$ such that $f_{\sigma_0} = \chi_{\sigma_0}$ and $f_\sigma=0$ for $\sigma\neq \sigma_0$ for some fixed $\sigma_0$ maps to $1\in K$.

The map $\pi_{\Delta}$ is clearly a homomorphism of $K[t_1,\ldots, t_n]$-modules: when we multiply $f= \{f_\sigma \}$ with $\theta_i$, we multiply each $f_\sigma(t_1,\ldots,t_n)$ with $t_i$. The map decreases degree by $n$ because all $\chi_\sigma$ are homogeneous polynomials of degree $n$. It remains to show that the image of $\pi_{\Delta}$ lies in the polynomial ring.

Each $\chi_\sigma$ is a product of linear functions that vanish on the $n$ hyperplanes in $V$ spanned by the facets of $\sigma$. This implies that the rational function $f_\sigma/\chi_\sigma$ can have at worst simple poles along these hyperplanes. Fix one $(n-2)$-dimensional simplex $\tau$ and let $H$ be the hyperplane it spans. Let $\sigma_1$ and $\sigma_2$ be the two $(n-1)$-dimensional simplices containing $\tau$. Now it suffices to prove that the residues of  $f_{\sigma_1}/\chi_{\sigma_1}$ and $f_{\sigma_2}/\chi_{\sigma_2}$ along the hyperplane $H$ sum to zero. This implies that all poles cancel and the image of $\pi_{\Delta}$ is a polynomial.

Consider $\chi_{\sigma_1} =  (\det \sigma_1) x_{j_1} x_{j_2} \cdots x_{j_n}$, where we take $t=(\det \sigma_1) x_{j_1}$ as the parameter that vanishes on $H$. The residue of $f_{\sigma_1}/\chi_{\sigma_1}$ with respect to the parameter $t$ is then 
\[ \frac{f_{\sigma_1}}{x_{j_2} \cdots x_{j_n}} \pmod t.\]
Working mod $t$ means that we restrict the rational function to the hyperplane $H$.

Now consider $\chi_{\sigma_2} = (\det \sigma_2) x_{j_1'} x'_{j_2} \cdots x'_{j_n}$, where $(\det \sigma_2) x_{j_1'}$ vanishes on $H$ and $x'_{j_2}, \ldots, x'_{j_n}$ are equal to $x_{j_2}, \ldots, x_{j_n}$ when restricted to the hyperplane $H$. From the normalization condition
\begin{align*}
  (\det \sigma_1) x_{j_1}\wedge x_{j_2} \wedge \cdots \wedge x_{j_n} &= -(\det \sigma_2) x_{j_1'}\wedge x'_{j_2} \wedge \cdots \wedge x'_{j_n} \\
    &= -(\det \sigma_2) x_{j_1'}\wedge x_{j_2} \wedge \cdots \wedge x_{j_n},
\end{align*}
it follows that 
\[ t= (\det \sigma_1) x_{j_1} = - (\det \sigma_2) x_{j_1'}.\]
The residue of $f_{\sigma_2}/\chi_{\sigma_2}$ with respect to the parameter $t$ is
\[ -\frac{f_{\sigma_2}}{x'_{j_2} \cdots x'_{j_n}} \equiv - \frac{f_{\sigma_2}}{x_{j_2} \cdots x_{j_n}} \pmod t.\]
This is equal to the negative of the residue of $f_{\sigma_1}/\chi_{\sigma_1}$ because $f_{\sigma_1}$ and $f_{\sigma_2}$ restrict to the same polynomial on $H$.
\end{proof}

\begin{rema} \label{rem-eval}
The map $\pi_{\Delta}$ can also be viewed as an evaluation map on piecewise polynomial functions. Choose a point $v_0\in V$ general enough such that $\chi_\sigma(v_0)\neq 0$ for any $\sigma$. We may now represent an element $f\in \cA^n(\Delta)$ as a vector of values $(f_\sigma(v_0))_\sigma \in K^M$. The map $\pi_{\Delta}$ is then defined as a weighted sum of these values:
\[  ( f_\sigma(v_0) )_\sigma \longmapsto \sum_\sigma \frac{f_\sigma(v_0)}{\chi_\sigma(v_0)}.\]
Expressing cohomology classes as vectors of values is a special case of a theorem by Carrell and Lieberman \cite{CarrellLieberman}.
\end{rema}

\subsection{Connected sums} \label{sec-connected-sum}

Consider the decomposition of an $(n-1)$-dimensional pseudo-manifold $\Delta$ as the connected sum of two pseudo-manifolds
\[ \Delta = \Delta_1 \#_D \Delta_2.\]
Here $D$ is a common $(n-1)$-dimensional subcomplex of $\Delta_1$ and  $\Delta_2$ that is a pseudo-manifold with boundary. We remove the interior of $D$ from $\Delta_1, \Delta_2$, and glue the remaining complexes along their common boundary. We assume that $\Delta$, $\Delta_1$ and $\Delta_2$ are oriented compatibly. This means that if a maximal simplex lies in both $\Delta$ and $\Delta_i$, then it has the same orientation in both.

  \begin{figure}[htb]
    \centering
    \includegraphics[width=0.8\textwidth]{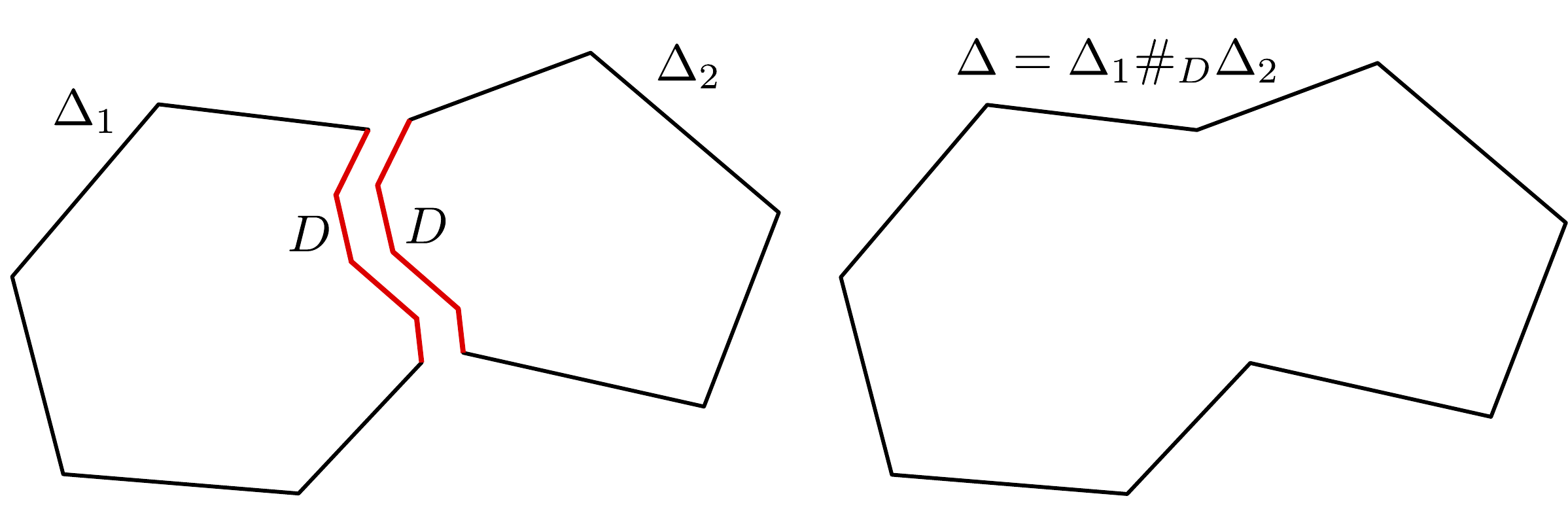}
    \caption{Connected sum of $1$-dimensional spheres $\Delta_1$ and $\Delta_2$ along $D$.}
  \end{figure}
  
  Let us also denote the simplicial complexes $\Delta_1$ and  $\Delta_2$ glued along $D$ by 
\[ \tilde{\Delta} = \Delta_1 \cup_D \Delta_2.\]
Let $\tilde{\theta}_i \in \cA^1(\tilde\Delta)$ be linear parameters for $\tilde{\Delta}$. These parameters, viewed as piecewise linear functions on $\tilde{\Delta}$, restrict to linear parameters on $\Delta$, $\Delta_1$ and $\Delta_2$. Similarly, a piecewise polynomial function $\tilde{f} \in \cA^n(\tilde{\Delta})$ restricts to piecewise polynomial functions $\tilde{f}|_\Delta \in \cA^n(\Delta)$,  $\tilde{f}|_{\Delta_1} \in \cA^n(\Delta_1)$ and $\tilde{f}|_{\Delta_2} \in \cA^n(\Delta_2)$. The latter two agree on $D$.

\begin{lemma} 
Let $\tilde{f} \in \cA^n(\tilde{\Delta})$. Then
\[ \pi_{\Delta} (\tilde{f}|_\Delta) = \pi_{\Delta_1}(\tilde{f}|_{\Delta_1}) + \pi_{\Delta_2}(\tilde{f}|_{\Delta_2}).\]
\end{lemma} 

\begin{proof} 
The maximal simplices of $D$ appear in $\Delta_1$ and $\Delta_2$ with opposite orientations. Hence these terms cancel on the right hand side. The remaining terms give the left hand side.
\end{proof}

\begin{rema} 
The  previous lemma was used in \cite{Karu, BL2, BBFK2}. Its meaning as integration over a connected sum was realized by Karl-Heinz Fieseler. The lemma says that Brion's integration map behaves like ordinary integration. One can decompose the domain of integration into pieces and sum the integrals over the pieces.
\end{rema}

We next consider a more general connected sum. Let $v_0$ be a new vertex and let 
\[ C(\Delta) = \{v_0\} * \Delta \]
be the cone over $\Delta$ with vertex $v_0$.  Let $\pi_i = \{v_0\} * \sigma_i$, $i=1,\ldots,M$ be the maximal simplices in $C(\Delta)$, and let $\Pi_i = \partial \pi_i$ be the simplicial $(n-1)$-spheres. Then
\[ \Delta = \#_{i=1}^M \Pi_i.\]
Here we use a more general notion of connected sum. We assume that $\Delta$ and $\Pi_i$ are oriented compatibly. Then an $(n-1)$-simplex $\sigma \in \Delta$ appears in the disjoint union $\sqcup_{i=1}^M \Pi_i$ exactly once and with the same orientation as in $\Delta$. All other $(n-1)$-simplices of  $\sqcup_{i=1}^M \Pi_i$ appear there twice with opposite orientations.

  \begin{figure}[htb]
    \centering
    \includegraphics[width=0.8\textwidth]{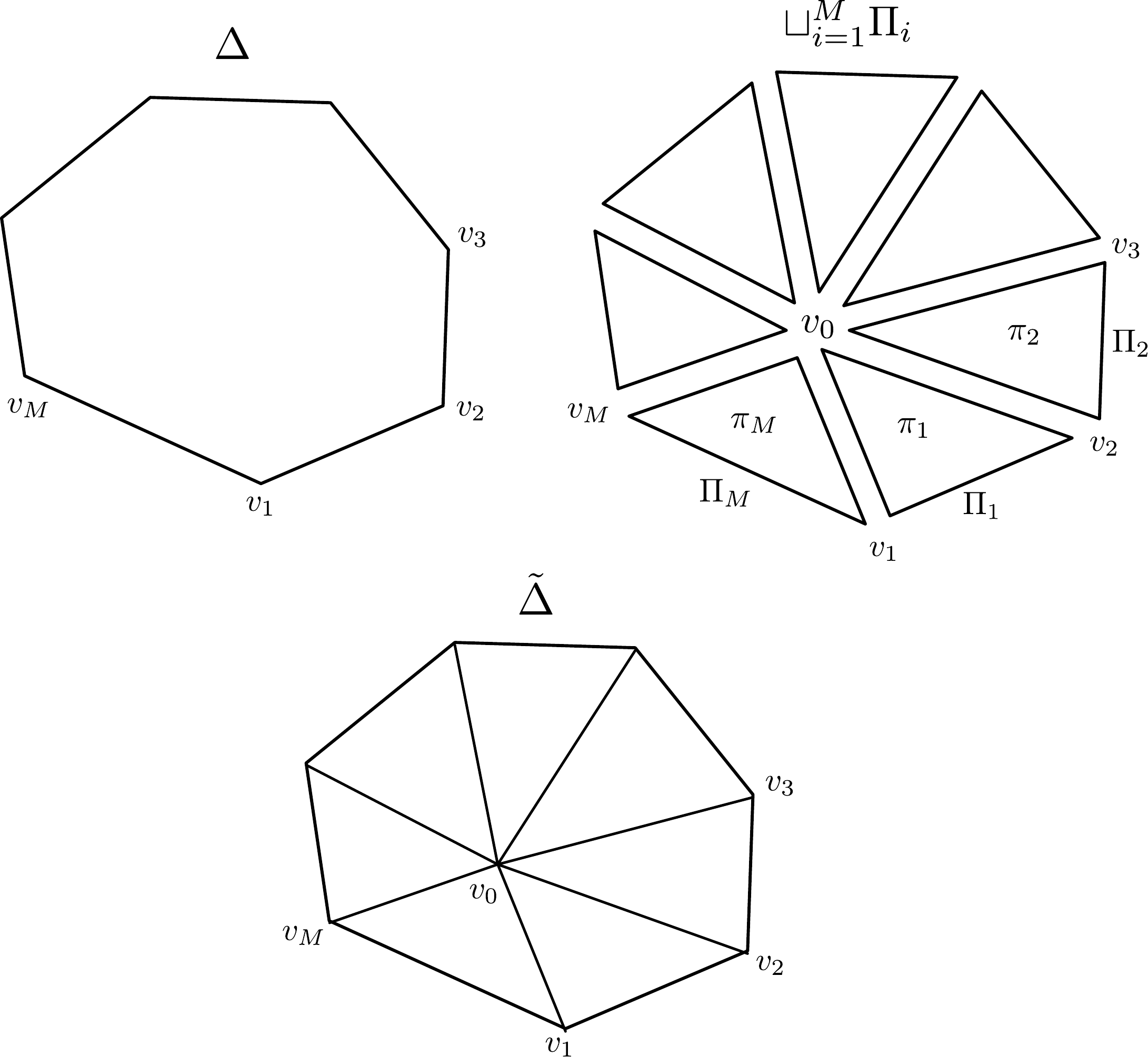}
    \caption{Decomposition of a $1$-sphere as a connected sum.}
  \end{figure}

As before, we let $\tilde{\Delta}$ be the union of $\Pi_i$. A system of linear parameters $\tilde{\theta}_i$ on $\tilde{\Delta}$ restricts to a system of parameters on $\Delta$ and all $\Pi_i$. 

\begin{lemma} \label{lem-int-sum}
Let $\tilde{f}\in \cA^n(\tilde{\Delta})$. Then 
\[ \pi_{\Delta}(\tilde{f}|_\Delta)  = \sum_{i=1}^M \pi_{\Pi_i}(\tilde{f}|_{\Pi_i}).  \]
\end{lemma}

The parameters $\tilde{\theta}_i$  have extra variables $a_{i,0}$ corresponding to the new vertex $v_0$. We may include these in the field $K$,
\[ K = k(\a{i, j})_{i=1,\ldots,n; j=0,\ldots, N}.\]  
However, the  map $\pi_{\Delta}$ does not depend on the variables $\a{i, 0}$. If $f\in \cA^n(\Delta)$ has coefficients in $k(\a{i, j})_{i=1,\ldots,n; j=1,\ldots, N}$ then $\pi_{\Delta}(f)$ also lies in the same field.

An alternative connected sum decomposition would be to take one of the existing vertices, say $v_1$, as the cone point and replace $C(\Delta)$ with 
\[ \{v_1\}* (\Delta \setminus \Star^\circ v_1).\]

  \begin{figure}[htb]
    \centering
    \includegraphics[width=0.8\textwidth]{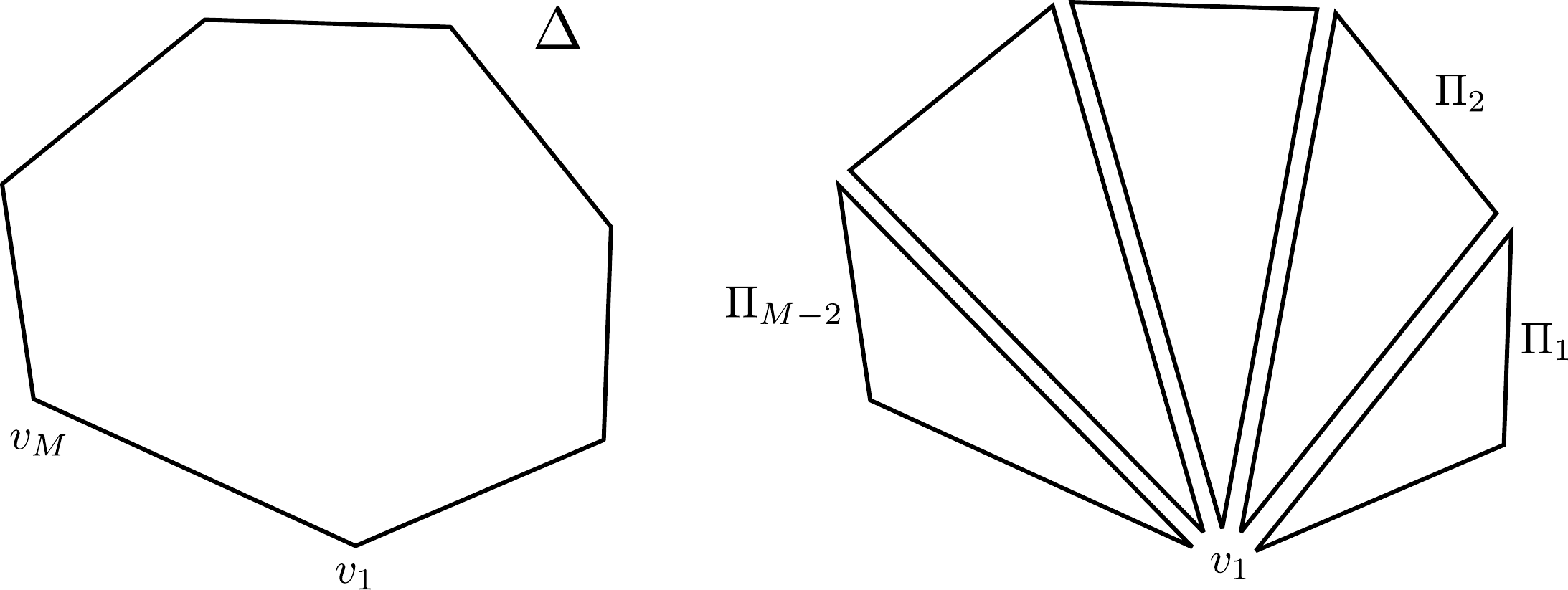}
    \caption{Alternative decomposition of a $1$-sphere as a connected sum.}
  \end{figure}

\section{Mixed volumes}
\label{sec-mixed}

Let $H$ be a standard graded, Artinian, Gorenstein $K$-algebra of socle degree $n$,
\[ H = K[x_1,\ldots,x_N]/I. \]
It is well-known that $H$ is determined by the linear function
\[  W: K[x_1,\ldots,x_N]_n \to H^n \stackrel{\isom}{\longrightarrow} K.\]
(We have denoted by subscript $n$ the degree $n$ homogeneous part of $K[x_1,\ldots,x_N]$.) Indeed,  one recovers the ideal $I$ from $W$ using the property that $f\in K[x_1,\ldots,x_N]_m$ lies in $I$ if and only if $W(fg) = 0$ for any $g$ of degree $n-m$. More generally, any nonzero linear function $W: K[x_1,\ldots,x_N]_n \to K$ determines a standard graded, Artinian, Gorenstein $K$-algebra $H$ of socle degree $n$.

For an oriented pseudo-manifold $\Delta$, let the function $W = W_\Delta$ be the composition
\[ W_\Delta:  K[x_1,\ldots,x_N]_n \to \cA^n(\Delta) \stackrel{\pi_\Delta}{\longrightarrow} K.\]
When $\Delta$ is a homology sphere over $K$, then $W_\Delta$ determines the algebra $H(\Delta)$. When $\Delta$ is an oriented pseudo-manifold over $K$, then $W_\Delta$ determines an algebra that we denote $\overline{H}(\Delta)$. This algebra in general is a quotient of the algebra $H(\Delta)$.

In the theory of polytopes and toric varieties the function $W_\Delta$ is known as the mixed volume.

\subsection{The case of $\Pi$} \label{sec-Pi}
Let $\pi$ be an $n$-simplex and $\Pi = \partial\pi$ the $(n-1)$-dimensional sphere. We compute here the mixed volume $W_\Pi$. 

Let $v_0, v_1,\ldots, v_n$ be the vertices of $\Pi$, and $\sigma_j = \{v_0,\ldots, \hat{v}_j, \ldots, v_n\}$ the maximal simplices. We choose the orientation on $\Pi$ so that $v_1,\ldots,v_n$ is positively oriented on the simplex $\sigma_0$.   Denote by 
\[ A = (\a{i, j})_{i,j} \]
the $n\times (n+1)$ matrix of variables, where the columns are indexed by $0,1,\ldots, n$ and the rows by $1,\ldots,n$. Let $X_j \in K$ be $(-1)^j$ times the determinant of the matrix $A$ with its $j$-th column removed. Then $X_j = \det \sigma_j$ as defined in Equation~(\ref{eq-dets}) on page~\pageref{eq-dets}.

\begin{lemma} \label{lem-Pi}
Let $f\in K[x_0,\ldots, x_n]_n$. Then
\[ W_\Pi \big( f(x_0,x_1,\ldots,x_n) \big) =   \frac{f(X_0, X_1, X_2, X_3,\ldots, X_n)}{X_0 X_1\cdots X_n}.\]
\end{lemma}

\begin{proof}
We first check that $W_\Pi(\theta_i g) = 0$ for any $g$ of degree $n-1$ and $i=1,\ldots,n$. It suffices to show that $\theta_i$ evaluated at $X_0,\ldots, X_n$ is zero. From the definition,
\[ \theta_i(X_0,\ldots, X_n) = \sum_{j=0}^n  \a{i, j} X_j.\]
This sum is the expansion of the determinant of the matrix $A$ with a copy of its $i$-th row added as the first row. Since the matrix has two repeated rows, its determinant is zero. 

The previous argument shows that the map $W_\Pi$ factors through $H^n(\Delta)$.  Let us check that its value on the monomial $\chi_{\sigma_0} = (\det\sigma_0) x_1\cdots x_n$ is $1$ as required:
\[ W_\Pi((\det\sigma_0)x_1\cdots x_n) = \frac{X_0 X_1 \cdots X_n}{X_0 X_1\cdots X_n} = 1. \qedhere\]
\end{proof}

To simplify notation, let us write the mixed volume as
\[ W_\Pi = \frac{1}{c_\Pi} \ev_\Pi,\]
where $c_\Pi = X_0 X_1 \cdots X_n \in K$ and $\ev_\Pi: K[x_0,\ldots, x_n] \to K$ is the evaluation map that sets $x_j= X_j$. The evaluation map is a $K$-algebra homomorphism.

Recall that in Section~\ref{sec-connected-sum} we decomposed an oriented pseudo-manifold $\Delta$ as a connected sum
\[ \Delta = \#_{i=1}^M \Pi_i.\]
The following result now follows from Lemma~\ref{lem-int-sum} and Lemma~\ref{lem-Pi}:

\begin{theorem} \label{thm-W-sum}
Let $\Delta$ be an oriented pseudo-manifold. Then
\[ W_\Delta = \sum_{i=1}^M W_{\Pi_i} = \sum_{i=1}^M \frac{1}{c_{\Pi_i}} \ev_{\Pi_i}.\]
\end{theorem}

In the theorem the map $\ev_{\Pi_i}$ acts on $K[x_1,\ldots,x_N]$ as a composition
\[ K[x_1,\ldots, x_N] \to K[x_{j_1},\ldots, x_{j_n}] \stackrel{\ev_{\Pi_i}}{\longrightarrow} K,\]
where the first map, the restriction to $\Pi_i$, sets $x_j=0$ if $v_j$ does not lie in $\Pi_i$. When $v_j = v_{j_l}$ is a vertex of $\Pi_i$, then $\ev_{\Pi_i}$ maps $x_j$ to $X_j$. However, the constants $X_j\in K$ depend not only on $j$ but also on all vertices of $\Pi_i$ and the orientation on $\Pi_i$.

\begin{rema} 
It is not too difficult to see that the previous theorem is nothing more than the integration $\pi_\Delta$ viewed as an evaluation map (see Remark~\ref{rem-eval}). Indeed, when we evaluate the summands of the map $\pi_\Delta$ (formula (\ref{eq-int}) on page \pageref{eq-int}) at the generic point $v_0$, we get the summands in the theorem.
\end{rema}

\subsection{The quadratic form $Q_l$}
Let $\Delta$ be an oriented pseudo-manifold of dimension $n-1$ over $K$, and let  $l=x_1+x_2+\cdots+x_N$. We define the quadratic form $Q_l$ on $K[x_1,\ldots, x_N]_m$ for $m\leq n/2$:
\[ Q_l(g) =  W_\Delta(l^{n-2m}g^2).\]
This form descends to a form on $H^m(\Delta)$, and in the case where $\Delta$ is not a homology sphere, to a form on the quotient space $\overline{H}^m(\Delta)$.

\begin{theorem} \label{thm-Q}
The quadratic form $Q_l$ on $K[x_1,\ldots, x_N]_m$ is
\[ Q_l (g) = \sum_{i=1}^M W_{\Pi_i} (l^{n-2m} g^2) = \sum_{i=1}^M \frac{1}{c_{\Pi_i}} \big[\ev_{\Pi_i} (l)\big]^{n-2m} \big[\ev_{\Pi_i} (g)\big]^2.\]
\end{theorem}
\begin{proof}
The second equality follows from the fact that the evaluation maps are $K$-algebra homomorphisms.
\end{proof}

The  theorem provides a diagonalization of the quadratic form $Q_l$.  Each map $\ev_{\Pi_i}$ defines a linear function on $K[x_1,\ldots, x_N]_m$. Let us call this function $z_i$. The quadratic form $Q_l$ is then 
\[ \sum_i d_i z_i^2,\]
where the coefficients are
\[ d_i = \frac{\big[\ev_{\Pi_i} (l)\big]^{n-2m}}{c_{\Pi_i}} \in K.\]
This expression of the quadratic form holds in any characteristic. It can be used, for example, to specialize the form from characteristic zero to characteristic $p$, assuming that $\Delta$ is oriented the same way in both characteristics. All coefficients in the form (the numerator and denominator of $d_i$, the coefficients of $z_i$) are polynomials in the variables $\a{i,j}$ with integer coefficients. If $g\in \ZZ[\a{i,j}][x_1,\ldots, x_N]_m$ is a polynomial such that $Q_l(g) = 0$ in $\QQ(\a{i,j})$, then $Q_l(\overline{g})= 0$ in $\FF_p(\a{i,j})$, where $\overline{g} = g \pmod p$.

The summands of the quadratic form $Q_l$ in Theorem~\ref{thm-Q} are defined over the field $K$ that includes the variables $\a{i, 0}$. However, the form itself does not depend on these variables and can be defined over the field $k(\a{i, j})_{i=1,\ldots,n; j=1,\ldots,N}$. The anisotropy of the form does not depend on which of the two fields we use.

\section{The conjecture of Papadakis and Petrotou}
\label{sec-conjecture}

We assume that the field $K=k(\a{i, j})$ has characteristic $2$ throughout this section. Papadakis and Petrotou study the values of the quadratic form $Q_l$ in $K$ and partial derivatives of these values with respect to $\a{i, j}$. 

Consider partial derivatives $\partial_{\a{i, j}}$ acting on $K=k(\a{i, j})$. Because of the characteristic $2$ assumption, these derivatives satisfy for any $f,g\in K$
\[ \partial_{\a{i, j}}^2 f=0, \quad \partial_{\a{i, j}} f^2=0, \quad \partial_{\a{i, j}} f^2 g= f^2 \partial_{\a{i, j}} g.\]

We will use capital letters $I, J, L$ to denote vectors of non-negative integers. Let $|J|$ be the number of components in the vector $J$. For $I=(i_1,\ldots,i_n)$ with $n$ components we let
\[ \partial_I = \partial_{\a{1, i_1}}  \partial_{\a{2, i_2}} \cdots \partial_{\a{n, i_n}}.\]
For $J=(j_1,\ldots, j_s)$, let $x_J$ be the degree $s$ monomial
\[ x_J =  x_{j_1} x_{j_2} \cdots x_{j_s}.\]
Note that $I$ and $J$ may contain repeated elements and $s$ may be larger than $n$.
There is some redundancy in this notation because $x_J$ only depends on $J$ up to permutation of components. However, $\partial_I$ does depend on the order of components in $I$.  
We write $\sqrt{x_J}$ for the monomial whose square is $x_J$ if such a monomial exists.

The following was stated in \cite{PapadakisPetrotou} as Conjecture 14.1 in case of homology spheres $\Delta$. We generalize it to the case of pseudo-manifolds, which by the characteristic $2$ assumption are automatically oriented.

\begin{theorem}[Conjecture of Papadakis and Petrotou] \label{thm-PPconj}
Let $\Delta$ be a  pseudo-manifold of dimension $n-1$ over $K$. For any integer vectors $I, J$ with $n$ components
\[ \partial_I W_\Delta(x_J) = \begin{cases} 
(W_\Delta (\sqrt{x_I x_J}))^2 & \text{if $\sqrt{x_I x_J}$ exists,} \\
0 & \text{otherwise.}
\end{cases}
\]
\end{theorem}

We will prove the conjecture below. Let us first see that it implies Theorem~\ref{thm-lef}. The argument here is similar to the proof of Theorem~\ref{thm-PP} in \cite{PapadakisPetrotou}. In fact, it is very natural to extend this theorem to the case of pseudo-manifolds $\Delta$ as in \cite{APP}. Recall that in Section~\ref{sec-mixed} we defined for any oriented pseudo-manifold $\Delta$ the algebra $\overline{H}(\Delta)$. This is equal to the algebra $H(\Delta)$ if $\Delta$ is a homology sphere.

\begin{coro} \label{cor-PPconj}
Let $h \in K[x_1,\ldots, x_N]_m$ for  some $0 \leq m \leq n/2$, and let $I, J$ have $n, n - 2m$ components, respectively. Then
\[ \partial_I W_\Delta(h^2 x_J) = (W_\Delta(h \sqrt{x_I x_J}))^2 \]
if $\sqrt{x_I x_J}$ exists, and is otherwise zero.
\end{coro}

\begin{proof}
Write $h$ as a linear combination of monomials, $h=\sum_L b_L x_L$. Then
\[ W_\Delta(h^2 x_J) = W_\Delta ( \sum_L b_L^2 x_L^2 x_J) = \sum_L b_L^2 W_\Delta (x_L^2 x_J).\]
Applying the derivative $\partial_I$ to this and using  Theorem~\ref{thm-PPconj}, we get
\[  \sum_L b_L^2 \partial_I W_\Delta (x_L^2 x_J) = \sum_L b_L^2 (W_\Delta (x_L \sqrt{x_I x_J}))^2 = (W_\Delta ( \sum_L b_L x_L \sqrt{x_I x_J}))^2\]
if the square root exists, and zero otherwise.
\end{proof}

The previous corollary shows why Theorem~\ref{thm-PPconj} is well suited for proving anisotropy theorems in characteristic $2$. The expression $W_\Delta(h \sqrt{x_I x_J})$ on the right hand side is the \Po pairing between $h$ and $\sqrt{x_I x_J}$. If $W_\Delta(h^2 x_J)$ on the left hand side is zero, then $h$ is orthogonal to $\sqrt{x_I x_J}$ for all $I$.

\begin{coro} \label{cor-reduce}
Let $\Delta$ be a pseudo-manifold of dimension $n-1$ over $K$. Consider $g\in K[x_1,\ldots, x_N]_m$ such that 
\[ l^p g^2 = 0 \quad \text{in $\overline{H}^{p+2m}(\Delta)$}\]
for some  $0\leq p \leq n-2m$. 
Let $q = \lfloor \frac{p}{2}\rfloor$. Then
\[ l^q g = 0 \quad \text{in $\overline{H}^{q+m}(\Delta)$.}\]
\end{coro}

\begin{proof}
If $l^p g^2 = 0$ in $\overline{H}^{p+2m}(\Delta)$, then for any integer vectors $I$ and $J$ with respectively $n$ and $n-2m-p$ components,
\[ \partial_I W_\Delta(l^p g^2 x_J) = 0. \]
If $p$ is even, then $p/2 = q$ and by Corollary~\ref{cor-PPconj},
\[ \partial_I W_\Delta(l^p g^2 x_J) = \partial_I W_\Delta((l^q g)^2 x_J) = (W_\Delta( l^q g \sqrt{x_I x_J}))^2 = 0 \]
if the square root exists, and so
\[ W_\Delta( l^q g \sqrt{x_I x_J}) = 0. \]
Since $I$ and $J$ can be chosen such that $x_I x_J = x_L^2$  for any of the monomials $x_L$ generating $\overline{H}^{n-q-m}(\Delta)$,  we conclude that $l^q g = 0$ in $\overline{H}^{q+m}(\Delta)$.

If $p$ is odd, then $l^p g^2 = (l^q g)^2 \cdot l = (l^q g)^2 \sum_{i=1}^N x_i$. Applying Corollary~\ref{cor-PPconj}, we have
\[ \partial_I W_\Delta(l^p g^2 x_J) = \sum_{i=1}^N  (W_\Delta(l^q g \sqrt{ x_i x_I x_J} ))^2 \]
if the square root exists. Let $L$ have $n-q-m$ components, and let $j$ be one of these components of $L$. Choose $I$ and $J$ such that $x_I x_J = x_L^2/x_j$. Then $\sqrt{ x_j x_I x_J}  = x_L$ and $\sqrt{ x_i x_I x_J}$ does not exist when $i \neq j$, so in this case
\[ \partial_I W_\Delta(l^p g^2 x_J)
 =   (W_\Delta(l^q g x_L))^2 = 0. \]
Hence,
\[ W_\Delta(l^q g x_L) = 0 \]
for any $x_L$. We conclude again that $l^q g = 0$ in $\overline{H}^{q+m}(\Delta)$.
\end{proof}

We now prove a generalization of Theorem~\ref{thm-lef} in the case of pseudo-manifolds.

\begin{theorem}
Let $\Delta$ be a pseudo-manifold of dimension $n-1$ over $K$. Then the quadratic form $Q_l$ on $\overline{H}^m(\Delta)$ is anisotropic  for any $m\leq n/2$.
\end{theorem}

\begin{proof}
Let $g \in \overline{H}^m(\Delta)$ be an isotropic element for $Q_l$,
\[ l^{n-2m} g^2 = 0.\]
Corollary~\ref{cor-reduce} implies that $l^q g = 0$, where $q = \lfloor \frac{n-2m}{2}\rfloor$. In particular, $l^q g^2 = 0$. Continuing this way we reduce the power of $l$ to zero and hence $g=0$.
\end{proof}

The rest of this section consists of the proof of Theorem~\ref{thm-PPconj}.

\subsection{Reductions} We start by reducing Theorem~\ref{thm-PPconj} to simpler cases. First notice that all expressions in Theorem~\ref{thm-PPconj} are defined over the field $\FF_2$. Hence we may assume that $k=\FF_2$. 

Recall that we wrote $W_\Delta = \sum_i W_{\Pi_i}$ in Theorem~\ref{thm-W-sum}. 

\begin{lemma} \label{lem-redPi}
Theorem~\ref{thm-PPconj} for all $\Pi_i$ implies it for $\Delta$.
\end{lemma}

\begin{proof}
This follows directly from the statement of the theorem using the characteristic $2$ assumption and the  observation that if a monomial $x_Ix_J$ restricts to a nonzero monomial on $\Pi_i$, then the monomial is a square if and only if its restriction is a square.
\end{proof}

From now on we will assume that $\Delta = \Pi$ as in Section~\ref{sec-Pi}. Assume that $\Pi$ has vertices $v_0, v_1, \ldots, v_n$. The matrix $A=(\a{i, j})$ has size $n\times (n+1)$, with columns indexed by $0,1,\ldots,n$ and rows by $1,\ldots,n$. We use the notation $X_j$, $j=0,1,\ldots,n$, for the determinant of $A$ with its $j$-th column removed. If $J= (j_1, \ldots, j_s)$ is a vector with entries in $\{0,1,\ldots,n\}$, we let
\[ X_J = X_{j_1} X_{j_2} \cdots X_{j_s}.\]

Theorem~\ref{thm-PPconj} for $\Pi$ can be further reduced to the following:

\begin{theorem} \label{thm-simplest}
Let $I$ and $J$ be vectors with entries in $\{0,1,\dots,n\}$. Assume that $|I|=n$ and $|J|$ is odd.  Then
  \[ \partial_I  X_J  = 
  \begin{cases}
    (X_L)^2 & \text{if $X_I X_J = X_L^2 X_{(0,1,\dots,n)}$}, \\
    0         &  \text{otherwise.}
  \end{cases}\]
\end{theorem}

\begin{lemma} \label{lem-conj-Pi}
Theorem~\ref{thm-simplest} implies Theorem~\ref{thm-PPconj} for $\Pi$.
\end{lemma}
\begin{proof}
Let $I$ and $J$ be as in the statement of Theorem~\ref{thm-PPconj}, and let $J'$ be such that $X_{J'} = X_J X_{(0,1,\dots,n)}$. Note that $|J'|=2n+1$ is odd. We claim that Theorem~\ref{thm-PPconj} for $I,J$ is equivalent to Theorem~\ref{thm-simplest} for $I, J'$. 

Using that 
\[ W_\Pi (x_{J}) = \frac{X_{J}}{c_\Pi}  = \frac{X_{J}}{ X_{(0,1,\dots,n)}}, \]
the statement of Theorem~\ref{thm-PPconj} for $\Pi, I,J$ is
  \[ \partial_I  \frac{X_J}{c_\Pi}  = 
  \begin{cases}
    \frac{X_I X_J}{c_\Pi^2} & \text{if $\sqrt{x_I x_J}$ exists}, \\
    0         &  \text{otherwise.}
  \end{cases}\]
The statement of Theorem~\ref{thm-simplest} for $I, J'$ is 
  \[ \partial_I  ( X_J c_\Pi )  = 
  \begin{cases}
    X_I X_J & \text{if $\sqrt{X_I X_J}$ exists}, \\
    0         &  \text{otherwise.}
  \end{cases}\]
These two equations differ by a factor of $c_\Pi^2$.
\end{proof}

We will prove Theorem~\ref{thm-simplest} below after some preparations.

\subsection{$\SL(n,k)$-invariance}

Let $A=(\a{i, j})$ be the $n\times (n+1)$ matrix of variables. For a matrix $B\in \SL(n,k)$, consider the linear change of variables from $A$ to $BA$. This defines an action of $\SL(n,k)$ on the polynomial ring $k[\a{i, j}]$. The first fundamental theorem of invariant theory for $\SL(n,k)$ states that if $k$ is an infinite field of any characteristic,  then the $k$-algebra of invariants under this action is generated by $X_0, X_1, \dots, X_n$. When  the field $k$ is finite, the same result holds if we consider absolute invariants. These are polynomials in $k[\a{i,j}]$ that are invariant under the action of $\SL(n,\overline{k})$, where $\overline{k}$ is the algebraic closure of $k$. The first fundamental theorem states that absolute invariants are again polynomials in $X_0, X_1, \dots, X_n$ with coefficients in $k$. (See \cite{Procesi2007}, Theorem~13.5.5 and the discussion of absolute invariants in  Section 13.6.1.)

\begin{lemma} \label{lem-invariance}
Let $I$ and $J$ be as in Theorem~\ref{thm-simplest}. Then the polynomial $\partial_I X_J \in k[\a{i, j}]$ is $\SL(n,k)$-invariant for any field $k$ of characteristic $2$. In particular, $\partial_I X_J \in \FF_2[\a{i, j}]$ is an absolute $\SL(n, \FF_2)$-invariant.
\end{lemma}

\begin{proof}
The group $\SL(n,k)$ is generated by elementary matrices. An elementary matrix acts on the matrix of variables by adding a constant $c$ times row $r$ to row $s$. It suffices to prove invariance under this change of variables. 

We may assume without loss of generality that $r=2$ and $s=1$. Consider the new variables
\[ a'_{i,j} = \begin{cases} 
\a{i, j} + c \a{2, j} & \text{if $i=1$,} \\ 
\a{i, j} & \text{otherwise.}
\end{cases}
\]
For a polynomial $f(a) = f(\a{i, j})$,  let us denote by $f(a')$ the result of substituting $\ap{i,j}$ in $\a{i, j}$. Similarly, let us write $\partial_I(a')$ for the partial derivative where we replace $\partial_{\a{i, j}}$ with $\partial_{a'_{i,j}}$. We need to prove that 
\[ (\partial_I X_J)(a') = (\partial_I X_J) (a). \]
We claim that if $\partial_I = \partial_{\a{1, i_1}}  \partial_{\a{2, i_2}} \partial_{\a{3, i_3}}\cdots \partial_{\a{n, i_n}}$, then 
\begin{equation} (\partial_I X_J)(a') = (\partial_I X_J)(a) - c (\partial_{\a{1, i_1}}  \partial_{\a{1, i_2}} \partial_{\a{3, i_3}} \cdots \partial_{\a{n, i_n}} X_J)(a). \label{eq-chain} \end{equation}
The next lemma shows that $\partial_{\a{1, i_1}}  \partial_{\a{1, i_2}} X_J = 0$, hence the second summand vanishes.

For any polynomial $f(\a{i, j})$, by replacing all symbols $a$ with $a'$, we have  
\[ (\partial_I f)(a')= \partial_I(a') f(a').\]
If $f$ is $\SL(n,k)$-invariant, then 
\[ f(a') = f(a) = f(a'_{1,j} - c a'_{2,j}, a'_{2,j},\ldots, a'_{n,j}).\] 
Let us now prove Equation~(\ref{eq-chain}). Since $X_J$ is $\SL(n,k)$-invariant, 
\[  (\partial_I X_J)(a') = \partial_I(a') X_J(a') = \partial_{\ap{1, i_1}}  \partial_{\ap{2, i_2}} \cdots \partial_{\ap{n, i_n}} X_J (a'_{1,j} - c a'_{2,j}, a'_{2,j},\ldots, a'_{n,j}).\]
Using the chain rule, this derivative is
\[ \big( \partial_I X_J - c \partial_{\a{1, i_1}}  \partial_{\a{1, i_2}}\partial_{\a{3, i_3}} \cdots \partial_{\a{n, i_n}} X_J \big) (a'_{1,j} - c a'_{2,j}, a'_{2,j},\ldots, a'_{n,j}).\]
Changing back to the variables $\a{i,j}$ gives the right hand side of  (\ref{eq-chain}).
\end{proof}

\begin{lemma} \label{lem-2deriv}
If $|J|$ is odd then $\partial_{\a{r, i_1}}  \partial_{\a{r, i_2}} X_J = 0$  for any $r, i_1, i_2$.
\end{lemma}

\begin{proof}
  It is enough to consider the case where $J$ contains no repeating indices, since any square factors of $X_J$ can be factored out of the partial derivatives. Under a suitable relabelling of rows and columns of $A$, we can assume that $r=1$, $i_1 = 1$ and $i_2 = 2$, so that the derivative under consideration is $\partial_{\a{1, 1}} \partial_{\a{1, 2}}  X_J$.
  
 Let us denote by $Y_{i,j} = Y_{j,i}$ the determinant of the matrix $A$ with its first row and columns $i,j$ removed. Then
 \[ \partial_{\a{1, i}} X_j = Y_{i,j}.\]
 The polynomials $X_i$ and $Y_{i,j}$ satisfy the following relations. For any distinct indices $i,j,p,q$
 \begin{equation} \label{eq-Plucker}
 Y_{i,j} Y_{p,q} - Y_{i,p} Y_{j,q} + Y_{i,q} Y_{j,p} = 0,
 \end{equation}
 and for any distinct indices $i,j,p$
 \begin{equation} \label{eq-flag}
 Y_{i,j} X_p - Y_{i,p} X_{j} + Y_{j,p} X_{i} = 0.
 \end{equation}
 These equations hold in any characteristic. In characteristic $2$ the signs in the equations are not important.
 The equations come from the Pl\"ucker embedding of the Grassmannian. Rows $2,3,\ldots,n$ of the  matrix $A$ span an $(n-1)$-plane in the $(n+1)$-space and hence define a point in the Grassmannian $\Gr(n-1, n+1)$. The polynomials $Y_{i,j}$ are the Pl\"ucker coordinates on this Grassmannian. These coordinates satisfy the Pl\"ucker relations in Equation~(\ref{eq-Plucker}). Similarly, the $n$ rows of the matrix $A$ define a point in $\Gr(n,n+1)$ with coordinates $X_i$. The relations in Equation~(\ref{eq-flag})  state that the $(n-1)$-plane with coordinates $Y_{i,j}$ lies in the $n$-plane with coordinates $X_i$.
 
Using the product rule we have
\[ \partial_{\a{1, 1}}  \partial_{\a{1, 2}} X_J = \sum \frac{X_J}{X_{i}X_{j}} Y_{1,i} Y_{2,j}.\]
Here the sum runs over all vectors $(i,j)$ where $i,j$ are distinct entries of $J$ such that $i\neq 1$ and $j\neq 2$. We claim that this sum is equal to 
\[ \sum \frac{X_J}{X_{i}X_{j}} Y_{1,2} Y_{i,j},\]
  where the sum now runs over all two element subsets $\{i,j\}$ of entries in $J$. To see this, first consider the case where $i$ and $j$ are both distinct from $1$ and $2$. In this case we apply the Pl\"ucker relation to get 
 \[  Y_{1,i} Y_{2,j} + Y_{2,i} Y_{1,j}= Y_{1,2} Y_{i,j}.\]
 The cases where $i=2$ or $j=1$ are simpler and do not require any relation.
 
 We are now reduced to proving that 
 \[  \sum_{\{i,j\}} \frac{X_J}{X_{i}X_{j}} Y_{1,2} Y_{i,j} = X_J Y_{1,2} \sum_{\{i,j\}} \frac{Y_{i,j}}{X_{i}X_{j}}  = 0.\]
Using Equation~(\ref{eq-flag}) we have for any distinct $i,j, p$
\[ \frac{Y_{i,j}}{X_{i}X_{j}} + \frac{Y_{i,p}}{X_{i}X_{p}} +\frac{Y_{j,p}}{X_{j}X_{p}} = 0.\]
Now consider all three element subsets $\{i,j,p\}$ of entries in $J$. Then 
\[ \sum_{\{i,j,p\}} \left( \frac{Y_{i,j}}{X_{i}X_{j}} + \frac{Y_{i,p}}{X_{i}X_{p}} +\frac{Y_{j,p}}{X_{j}X_{p}} \right) = 0.\]
Since every pair $\{i,j\}$ occurs in an odd number of triples $\{i,j,p\}$, this sum is equal to 
\[ \sum_{\{i,j\}} \frac{Y_{i,j}}{X_{i}X_{j}}. \qedhere\]
\end{proof}

\subsection{Proof of Theorem~\ref{thm-simplest}}

Lemma~\ref{lem-invariance} implies that $\partial_I X_J$ is a polynomial in $X_0, X_1, \cdots, X_n$ with coefficients in $\FF_2$.  Consider the grading by $\ZZ^{n+1}$ on the  ring $\FF_2[\a{i, j}]$ such that $\a{i, j}$ has degree $e_j$. Here $e_0,\ldots, e_{n}$ is the standard basis for $\ZZ^{n+1}$. Let $\one = (1,\ldots, 1)$. Then $X_i$, $i=0,\ldots, n$, is homogeneous with
\[ \deg X_i = \one - e_i.\]
Since the vectors $\one - e_i$ are linearly independent, there can be at most one monomial $X_L$ in each degree. 
The partial derivative $\partial_{\a{i,j}}$ applied to a homogeneous polynomial reduces its degree by $e_j$ (or is zero). Since $X_J$ is homogeneous, so is $\partial_I X_J$, hence $\partial_I X_J$ is equal to a constant $c$ times a monomial $X_M$. Here $c\in \FF_2$, hence $\partial_I X_J$ is either $X_M$ or $0$.
Computing the degrees, the monomial $X_M$ must satisfy $X_I X_J = X_M X_{(0,1,\dots,n)}$.
Theorem~\ref{thm-simplest} has two cases depending on whether $\sqrt{X_M}$ exists or not.

\begin{lemma} 
% If $\sqrt{X_M}$ does not exist then $\partial_I X_J = 0$.
If $\partial_I X_J \neq 0$, then $\sqrt{X_M}$ exists.
\end{lemma}

\begin{proof}
Suppose that $\partial_I X_J = X_M$. By Lemma~\ref{lem-2deriv}, all partial derivatives  
 \[ \partial_{\a{i, j}} X_M = \partial_{\a{i, j}} \partial_I X_J\]
 vanish. This implies that $X_M$ is the square of a polynomial in $\FF_2[\a{i, j}]$. Since $X_M$ is the product of irreducible polynomials $X_i$, it follows that $X_M$ must be the square of a monomial $X_L$.
\end{proof}

\begin{lemma}
% If $\sqrt{X_M} = X_L$ then $\partial_I X_J = X_M$.
If $\sqrt{X_M}$ exists, then $\partial_I X_J \neq 0$.
\end{lemma}

\begin{proof}
We will prove that $\partial_I X_J \neq 0$ by induction on $n$. 
The base of the induction is $n=0$. In this case $\partial_I X_J = X_J =1$.

Consider now $n>0$.  Let $I=(i_1,\ldots, i_n)$ and $J=(j_1,\ldots,j_s)$, where $s$ is odd. By assumption, there exists a monomial $X_L$ such that $X_I X_J = X_L^2 X_{(0,1,\ldots,n)}$.  To prove that $\partial_I X_J \neq 0$, we may factor out squares in $X_J$ and assume that $J$ has no repeated entries.

There exists an entry in $J$, say $j_1$, such that $j_1 \neq i_r$ for $r=1,\ldots, n$. This follows from the fact that $X_{(0,1,\ldots,n)}$ of degree $n+1$ divides $X_I X_J$, but $X_I$ has degree $n$. Define the monomial  
  \[ \mu = \a{1,i_1} \a{1,j_1}^{s-1}.\]
Then the coefficient of $\mu$ in $X_J=X_{j_1} \cdots X_{j_s}$ is 
\[  Y_{j_1, i_1} Y_{j_2, j_1} \cdots Y_{j_s,j_1}.\]
Indeed, $X_{j_1}$ does not contain $\a{1,j_1}$. Hence $\a{1,j_1}^{s-1}$ must come from the factors $X_{j_2}, \ldots, X_{j_s}$  and $\a{1,i_1}$ from the factor $X_{j_1}$. As before, we have denoted by $Y_{i,j}$ the coefficient of $\a{1,j}$ in $X_i$.

Notice that the coefficient of $\a{1,j_1}^{s-1}$ in $\partial_I X_J$ is 
\[ \partial_{\a{2,i_2}} \partial_{\a{3,i_3}} \cdots \partial_{\a{n,i_n}} Y_{j_1, i_1} Y_{j_2, j_1} \cdots Y_{j_s,j_1}.\]
If this coefficient is nonzero then also $\partial_I X_J$ is nonzero. We claim that this coefficient being nonzero follows by induction from the case of dimension $n-1$.  From the matrix $(\a{i,j})$ we have removed row $1$ and column $j_1$, the derivative $\partial_I$ is replaced with $\partial_{I'}$, where $I' = (i_2,\ldots, i_n)$, and $X_J$ is replaced (using a similar notation in dimension $n-1$) with $X_{J'}$, where $J'=(i_1, j_2, \ldots, j_s)$. 

Let us check that we can apply the induction assumption to prove that $\partial_{I'} X_{J'} \neq 0$. The vectors $I'$ and $J'$ satisfy
\[ X_{I'} X_{J'} = \frac{X_I}{X_{i_1}} \cdot \frac{ X_{i_1} X_J }{ X_{j_1} } = \frac{X_L^2 X_{(0,1,\dots,n-1,n)}}{X_{j_1}}  \]
Since we assumed that $J$ does not contain repeated entries, $X_{j_1}$ appears in the numerator of the fraction to the first power. If we suppose that $j_1 = n$, then
\[ X_{I'} X_{J'} = X_L^2 X_{(0,1,\dots, n-1)}, \]
where $X_{j_1} = X_n$ does not appear in any monomial. By induction, $\partial_{I'} X_{J'} \neq 0$.
\end{proof}

\section{Anisotropy in characteristic $0$}
\label{sec-consec}

In this section we prove Theorem~\ref{thm-2to0}. Let $K=\FF_2(\a{i, j})$ and $K_0 = \QQ(\a{i, j})$. We write $H(\Delta)_K$ and $H(\Delta)_{K_0}$ for the algebras defined over the fields $K$ and $K_0$, respectively. Similarly for $\cA(\Delta)_K$ and $\cA(\Delta)_{K_0}$.

\begin{lemma}
Let $\Delta$ be a homology sphere over $K$, and let $B$ be a set of monomials in $x_i$ that forms a basis for the vector space $H(\Delta)_K$. Then $B$ also forms a basis for $H(\Delta)_{K_0}$.
\end{lemma}

\begin{proof}
Since $\cA(\Delta)_K$ is a graded free $K[\theta_1,\ldots,\theta_n]$-module, the product of elements of $B$ with monomials in $\theta_i$ gives a basis for  $\cA(\Delta)_K$. We claim that the same set of products of elements of $B$ with monomials in $\theta_i$ is also a basis for $\cA(\Delta)_{K_0}$. For this it suffices to prove linear independence, because the dimension of $\cA(\Delta)$ in each degree is independent of the field.
If there is a relation between these elements with coefficients in $K_0$, we may clear denominators and assume that the coefficients lie in $\ZZ[\a{i, j}]$  so that not all coefficients are  divisible by $2$. Such a relation gives a nontrivial relation mod $2$. 

This proves that $\cA(\Delta)_{K_0}$ is a free $K_0[\theta_1,\ldots,\theta_n]$-module with basis $B$. Hence $B$ gives a basis for  $H(\Delta)_{K_0}$.
\end{proof}

\begin{proof}[Proof of Theorem~\ref{thm-2to0}]
Let $B_m$ be a basis of monomials for $H^m(\Delta)_K$ as in the lemma. Suppose $g\in H^m(\Delta)_{K_0}$ is nonzero and $Q_l(g)=0$. We may clear denominators and assume that $g$ is a linear combination of monomials in $B_m$ with coefficients in $\ZZ[\a{i, j}]$, not all coefficients divisible by $2$. This $g \pmod 2$ gives a nonzero element $\bar{g}\in H^m(\Delta)_{K}$. Moreover, $Q_l(\bar{g}) = Q_l(g) \pmod 2 = 0$. This contradicts Theorem~\ref{thm-lef}.
\end{proof}

Theorem~\ref{thm-2to0} does not extend to arbitrary orientable pseudo-manifolds $\Delta$. An example where the specialization argument fails is where $\Delta$ is a homology sphere over $\QQ$ but not over $\FF_2$. In this case $\Delta$ is still an orientable pseudo-manifold over $\FF_2$, but anisotropy for $\overline{H}(\Delta)_K$ does not imply anisotropy for $\overline{H}(\Delta)_{K_0} = H(\Delta)_{K_0}$.

\bibliographystyle{plain}
\bibliography{WL}{}

\end{document}